\documentclass[12 pt]{article}
\usepackage{amsmath}
\usepackage{amsfonts}
\usepackage{amssymb}
\usepackage{amsthm}
\usepackage{hyperref}
\usepackage{graphicx}
\usepackage{subcaption}
\usepackage{mathtools}
\usepackage{enumitem}
\usepackage[style=alphabetic,sorting=nyt,maxnames=5,maxalphanames=4,backend=bibtex]{biblatex}

\newtheorem{thm}{Theorem}[section]
\newtheorem{lem}[thm]{Lemma}
\newtheorem{conj}[thm]{Conjecture}
\newtheorem{cor}[thm]{Corollary}

\newcommand\numberthis{\addtocounter{equation}{1}\tag{\theequation}}
\DeclareMathOperator{\tr}{tr}
\DeclareMathOperator{\trans}{trans}
\DeclareMathOperator{\SL}{SL}
\DeclareMathOperator{\SU}{SU}
\DeclareMathOperator{\PSL}{PSL}
\DeclareMathOperator{\PSU}{PSU}

\title{Representations of the (-2,3,7)-Pretzel Knot and Orderability of Dehn Surgeries}
\author{Konstantinos Varvarezos}
\begin{document}
\maketitle

\begin{abstract}
We construct a 1-parameter family of $\SL_2(\mathbf{R})$ representations of the pretzel knot $P(-2,3,7)$.  As a consequence, we conclude that Dehn surgeries on this knot are left-orderable for all rational surgery slopes less than 6.  Furthermore, we discuss a family of knots and exhibit similar orderability results for a few other examples.
\end{abstract}

\section{Introduction}
This paper studies the character variety of a certain knot group in view of the relationship between $\widetilde{\PSL_2(\mathbf{R})}$ representations and the orderability of Dehn surgeries on the knot.  This is of interest because of an outstanding conjectured relationship between orderability and L-spaces.

A \textit{left-ordering} on a group $G$ is a total ordering $\prec$ on the elements of $G$ that is invariant under left-multiplication; that is, $g \prec h$ implies $fg \prec fh$ for all $f,g,h \in G$.  A group is said to be \textit{left-orderable} if it is nontrivial and admits a left ordering.  A 3-manifold $M$ is called \textit{orderable} if $\pi_1(M)$ is left-orderable.

If $M$ is a rational homology 3-sphere, then the rank of its Heegaard Floer homology is bounded below by the order of its first (integral) homology group.  $M$ is called an \textit{L-space} if equality holds; that is, if $\mathrm{rk}\big(\widehat{HF}(M)\big) = \left|H_1(M;\mathbf{Z})\right|$.

This work is motivated by the following proposed connection between L-spaces and orderability, first conjectured by Boyer, Gordon, and Watson.

\begin{conj}[\cite{BGW}] \label{conj:LS}
An irreducible rational homology 3-sphere is an L-space if and only if
its fundamental group is not left-orderable.
\end{conj}

In \cite{BGW}, this equivalence was shown to hold for all closed, connected, orientable,
geometric three-manifolds that are non-hyperbolic.  Knot groups are of particular interest due to the fact that knot complements are often hyperbolic and also because the L-space surgery interval for an L-space knot is known to be $[2g-1,\infty),$ where $g$ is the Seifert genus of the knot \cite{OSz}.

We shall primarily focus on the $(-2,3,7)$-pretzel knot, which is an L-space knot.  This knot has genus 5, and so if Conjecture \ref{conj:LS} holds, one would expect the orderable surgeries to be precisely from slopes in the interval $(-\infty, 9).$  It has been shown in \cite{Nie} that surgery along any slope outside this interval always yields a non-orderable manifold.  Moreover, Nie also showed that surgery along a slopes in $(-\epsilon,\epsilon)$ yields an orderable manifold for some sufficiently small $\epsilon>0.$  In this work, we improve the result to:
\begin{thm}\label{thm:main}
Let $X$ denote the exterior of the $(-2,3,7)$-pretzel knot.  Then $X(r)$ is orderable for all $r\in (-\infty, 6).$
\end{thm}
This leaves the slope interval $[6,9)$ as still unverified vis-\`{a}-vis Conjecture \ref{conj:LS}.

The proof of the main theorem follows the strategy developed by Culler and Dunfield in \cite{CD}.  In particular, we compute a certain one-parameter family of $\widetilde{\PSL_2(\mathbf{R})}$-representations of the knot group, and from that, we conclude the existence of a curve on the translation extension locus associated to such representations.  In fact, Culler and Dunfield used numerical methods to produce an image of the translation extension locus associated to the knot $P(-2,3,7)$  (see Figure 3 in \cite{CD}
).  Theorem \ref{thm:main} is precisely the result one expects from that image.


The $(-2,3,7)$-pretzel knot can be viewed as one of a family of \textit{twisted torus knots}.  In particular, consider the $(3,3k+2)$ torus knot with $m$ full twists about a pair of strands (pictured in Figure \ref{fig:knot}), which we denote by $T_{3,3k+2}^m$.  Notice that $T_{3,5}^1$ is the pretzel knot $P(-2,3,7).$  In section \ref{sec:gen}, we consider possible extensions of our result to the family $T_{3,3+2}^1$.

\begin{figure}
\centering
\includegraphics[width=.5\textwidth]{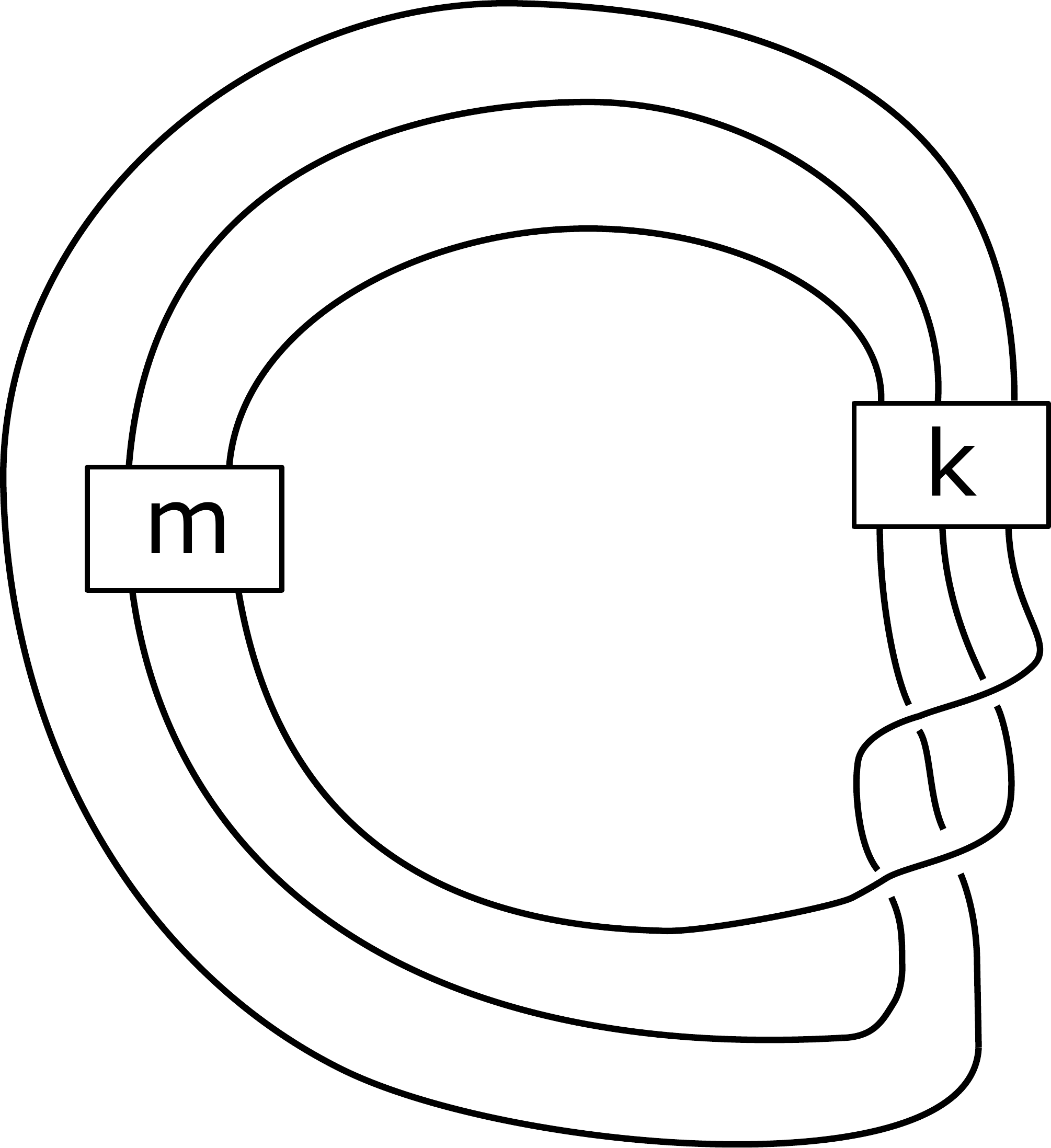}
\caption{The twisted torus knot $T^m_{3,3k+2}.$  Here the boxes represent $k$ and $m$ full twists of the strands passing through.}
\label{fig:knot}
\end{figure}

\subsection*{Acknowledgements}
The author would like to thank Professors Zolt\'{a}n Szab\'{o} and Peter Ozsv\'{a}th for suggesting and encouraging work on this problem.  This work was supported by the NSF RTG grant DMS-1502424.
\section{Background}
In what follows, we shall denote by $G$ the group $\PSL_2(\mathbf{R)}$ and by $\tilde{G}$ we shall denote its universal cover $\widetilde{\PSL_2(\mathbf{R})}$.

\subsection{Reducible Representations}\label{sec:red}

Let $K \subseteq S^3$ be a knot and denote by $X$ the knot exterior.  The abelianization map $\pi_1(X) \to H_1(X;\mathbf{Z})\cong \mathbf{Z}$ sends the meridian $\mu$ to a generator of the first homology.  For any $\zeta \in S^1 \subseteq \mathbf{C}$ (the complex unit circle), one can define a reducible representation $\rho_\zeta : \pi_1(X) \to \PSU(1,1) \cong G$ given by composing the abelianization map with:
\begin{align*}
H_1(X;\mathbf{Z}) &\to \PSU(1,1)\\
[\mu] &\mapsto \pm \left(\begin{matrix}
\zeta^{1/2} & 0\\
0 & \zeta^{-1/2}
\end{matrix}\right)
\end{align*}
where $\zeta^{1/2}$ is any square root of $\zeta$.

We are interested in finding paths of irreducible representations that are deformations of some such reducible representation.  The following theorem gives necessary and sufficient conditions for this to happen:

\begin{thm}[Lemma 4.8 and Proposition 10.3 of \cite{HP}]\label{thm:hp}
With $K$ and $\zeta$ as above, $\rho_\zeta$ may be deformed into a family of irreducible $\PSL_2(\mathbf{C})$ representations if and only if $\zeta$ is a root of the Alexander polynomial.  In that case, the irreducible curve of characters is contained in $\PSU(2)$ on one side and in $\PSU(1,1)\cong \PSL_2(\mathbf{R)}$ on the other side.
\end{thm}

\subsection{The Translation Extension Locus}
Following \cite{CD}, we give a description of the translation extension locus and explain how it may be used to construct left-orderings.  Let $\rho:\pi_1(X) \to \tilde{G}$ be a representation of the fundamental group of the knot exterior.  As $G$ acts on the circle via orientation-preserving transformations, $\tilde{G}$ acts on $\mathbf{R}$ via orientation-preserving transformations as well and, in fact, can be viewed as a subgroup of $\mathrm{Homeo}^+(\mathbf{R})$.  More precisely, we are viewing $\mathbf{R}$ as the universal cover of $S^1$ via the covering map $x \mapsto e^{2\pi i x}$ so that, for instance, lifts of the identity matrix act on $\mathbf{R}$ by integral translations.  We are interested in when such representations factor through the fundamental group of the filled manifold $X(r)$, for then, by \cite{BRW}, it will follow that $X(r)$ will be left-orderable.

Let us define the translation map $\trans:\tilde{G} \to \mathbf{R}$ by
\[
\trans(g) \coloneqq \lim_{n \to \infty} \frac{g^n(x)-x}{n}
\]
for any $x \in \mathbf{R}$ (note this definition does not depend on $x$).  We call an element of $\tilde{G}$ elliptic, parabolic, or hyperbolic if its image under the natural projection $p:\tilde{G} \to G$ is elliptic, parabolic, or hyperbolic, respectively.  Note that for any elliptic $g \in \tilde{G}$, $\trans(g)$ is equal to $\frac{1}{2\pi}$ times the rotation angle corresponding to the action of $p(g)$, up to additon of an integer.

Now let us define the translation extension locus.  For each representation $\rho:\pi_1(X) \to \tilde{G},$ we may consider its restriction to the peripheral subgroup: $\rho|_{\pi_1(\partial X)}: \pi_1(\partial X) \to \tilde{G}$.  As $\partial X$ is a torus, $\pi_1(\partial X) \cong \mathbf{Z}^2$ is abelian, and hence all (nontrivial) elements in the image of $\rho|_{\pi_1(\partial X)}$ are of the same type (elliptic, parabolic, or hyperbolic).  So we shall call $\rho$ elliptic, parabolic, or hyperbolic if the image of the restriction to the peripheral subgroup consists of elements of the respective type.  Given $\rho$ as above, we may consider the composition $\trans \circ \rho|_{\pi_1(\partial X)}: \pi_1(\partial X) \to \mathbf{R}$.  Because the peripheral subgroup is abelian, this is actually a homomorphism, and hence may be considered as an element of $\mathrm{Hom}(\pi_1(\partial X),\mathbf{R}) \cong H^1(\partial X, \mathbf{R}) \cong \mathbf{R}^2$.  If $(\mu, \lambda)$ are the natural meridian-longitude basis for $H_1(\partial X, \mathbf{R}),$ then one may consider the dual basis $(\mu^*,\lambda^*)$ for $H^1(\partial X, \mathbf{R})$.  Using this basis, we may consider the subset $E$ of $\mathbf{R}^2$ corresponding to all elliptic and parabolic representations $\rho: \pi_1(X) \to \tilde{G}.$  The \textit{translation extension locus} is the closure of this set.  We denote it $EL_{\tilde{G}}(X) \coloneqq \bar{E} \subseteq \mathbf{R}^2$.

Let $r = \frac{p}{q} \in \mathbf{Q}$ with $p,q$ relatively prime.  Then $\gamma = p \mu + q \lambda$ is a primitive element of $H_1(\partial X, \mathbf{Z})$ and is represented by a simple closed curve in $\partial X$.  Let $L_r$ denote the set of elements of $H^1(\partial X, \mathbf{R})$ which vanish on $\gamma$.  In the meridian-longitude-dual basis, this corresponds to the $(a,b)\in \mathbf{R}^2$ such that $ap+bq=0$.  Graphically, this is the line through the origin in $\mathbf{R}^2$ with slope $-\frac{p}{q} = -r$.  Suppose an element in $EL_{\tilde{G}}(X) \cap L_r$ comes from an elliptic representation $\rho$.  Then, $\trans(\rho(\mu^p\lambda^q))=0$.  As remarked above, the translation of an elliptic element corresponds to the rotation angle about its fixed point, and so, $\rho(\mu^p\lambda^q) = 1\in \tilde{G}.$  It follows that $\rho$ factors through $\pi_1(X)/\langle\langle \mu^p\lambda^q \rangle\rangle \cong \pi_1(X(r)).$  It follows by \cite{BRW} that $X(r)$ is orderable if it is irreducible.  In fact, Culler and Dunfield proved the following slightly more general result, for $M$ any rational homology solid torus:

\begin{thm}[Lemma 4.4 of \cite{CD}]\label{thm:int}
Suppose $M$ is a compact orientable irreducible 3-manifold with $\partial M$ a torus, and assume the Dehn filling $M(r)$ is irreducible.  If $L_r$ meets $EL_{\tilde{G}}(M)$ at a nonzero point which is not parabolic or ideal, then $M(r)$ is orderable.
\end{thm}

\begin{proof}[Remark.]\let\qed\relax Here an ideal point is an element of $\bar{E}\setminus E$ so that the hypotheses of the theorem may be equivalently expressed as $L_r$ intersecting $EL_{\tilde{G}}(M)$ at a nonzero elliptic point.
\end{proof}

Culler and Dunfield show that the translation extension locus has the following properties:

\begin{thm}[Theorem 4.3 of \cite{CD}]\label{thm:el}
The extension locus $EL_{\tilde{G}}(M)$ is a locally finite union of analytic arcs and isolated points. It is invariant under $D_{\infty}(M)$ with quotient homeomorphic to a finite graph. The quotient contains finitely many points which are ideal or parabolic in the sense defined above. The locus $EL_{\tilde{G}}(M)$ contains the horizontal axis $L_{\lambda}$, which comes from representations to $\tilde{G}$ with abelian image.
\end{thm}

Here $D_{\infty}(M)$ is a copy of the infinite dihedral group which acts on $\mathbf{R}^2$ via horizontal translations and rotations which depend on the homology of $M.$  In the case that $M$ is a knot exterior, $D_{\infty}(M)$ is generated by the unit horizontal translation 	$(x,y)\mapsto (x+1,y)$ and the rotation about the origin $(x,y)\mapsto (-x,-y)$.  In particular, the section of the translation extension locus with $0\leq x \leq \frac{1}{2}$ is a fundamental domain for the entire locus.

\subsection{The Knot Group}

As noted in the inroduction, the knot K can be interpreted as a twisted torus knot.  We shall use a presentation of the group $\pi_1(X)$ that comes from this interpretation, citing a result of Clay and Watson.

\begin{lem}[\cite{CW}]\label{lem:genpres}
If $X$ is the exterior of the twisted torus knot $T^m_{3,3k+2}$, then 
\[
\pi_1(X) = \left\langle a, b \left| a^2 \left(b^{-k}a\right)^m a = b^{2k+1} \left(b^{-k}a\right)^m b^{k+1} \right. \right\rangle
\]
Moreover, the peripheral subgroup is generated by the meridian $\mu$ and longitude $\sigma$ given in this presentation by:
\begin{align*}
\mu &= a^{-1} b^{k+1} \\
\sigma &= a \left(b^{-k}a\right)^m a \left(b^{-k}a\right)^m a
\end{align*}
The homological longitude $\lambda$ is then
\[
\lambda = \mu^{-3(3k+2)-4m} \sigma
\]
\end{lem}

\begin{proof}[Remark.]\let\qed\relax
This is essentially the content of Propositions 3.1 and 3.2 in \cite{CW} except that there, conjugates of $\mu$ and $\sigma$ by the generator $a$ are used instead.  Also note that there is a misprint in the paragraph following Proposition 3.2 in the formula for the homological longitude.  Indeed, by considering the abelianization, one finds that: $[a]^3=[b]^{3k+2}$ and moreover, $[a] = [\mu]^{3k+2}$ and $[b] = [\mu]^3$.
\end{proof}

As noted in the introduction, when $k=1$ we recover the pretzel knots $P(-2,5,2m+5)$.  So setting $m=1$ as well, we obtain:

\begin{cor}\label{cor:pres}
If $X$ is the exterior of the pretzel knot $P(-2,3,7)$, then 
\[
\pi_1(X) = \left\langle a, b \left| a^2 b^{-1} a^2 = b^2 a b^2 \right. \right\rangle
\]
Moreover, the peripheral subgroup is generated by the meridian $\mu$ and longitude $\sigma$ given in this presentation by:
\begin{align*}
\mu &= a^{-1} b^2 \\
\sigma &= a b^{-1} a^2 b^{-1} a^2
\end{align*}
The homological longitude $\lambda$ is then
\[
\lambda = \mu^{-19} \sigma
\]
\end{cor}

\section{Computing Representations}\label{sec:reps}

In our computation of representations, we follow the methods of \cite{Chen}, who computed the $\SL_2(\mathbf{C})$ character variety for even 3-stranded pretzel knots.  For our purposes, it is more convenient to apply those methods to this particular presentation rather than using his result directly.  We state here the relevant facts:

\begin{lem}\label{lem:chen}
If $X,Y \in \SL_2\mathbf{C}$ then 
\begin{enumerate}[label=(\alph*)]
\item $XYX = \tr(XY) X - Y^{-1}$;
\item $\tr(X^{-1}) = \tr(X)$;
\item $\tr(XY) = \tr(YX)$;
\item $\tr(X^k) = \omega_k(\tr(X))X-\omega_{k-1}(\tr(X))I$;
\item $X$ and $Y$ have no common eigenvector if and only if $I,X,Y,XY$ are linearly independent. \label{linind}
\end{enumerate}
\end{lem}
Here $I$ is the $2\times2$ identity matrix, and the $\omega_k$ are polynomials characterized by:
\begin{align*}
\omega_0(t) &\equiv 0 \\
\omega_1(t) &\equiv 1 \\
\omega_{k+1}(t) &= t \omega_k(t) - \omega_{k-1}(t)\\
\omega_{-k}(t) &= -\omega_k(t)
\end{align*}
In particular, we have:
\begin{align*}
X^2 &= \tr(X)X-I\\
X^{-1} &= \tr(X)I-X
\end{align*}

Let $\rho: \pi_1(X) \rightarrow \SU(1,1)$ be a representation for the knot group.  Put $A = \rho(a)$ and $B = \rho(b)$, where $a$ and $b$ are the generators of the group as in the presentation in Lemma \ref{cor:pres}.  Let us also put $t = \tr(A)$, $s = \tr(B)$, and $r = \tr(AB)$.  Using the lemma, we may compute:

\begin{align*}
A^2B^{-1}A^2 &= \tr(A^2B^{-1}) A^2 - (B^{-1})^{-1} \\
&= \tr((tA-I)(sI-B))(tA-I) - B \\
&= (ts\cdot \tr(A) -s\cdot \tr(I) - t\cdot \tr(AB) + \tr(B))(tA-I)-B\\
&=(t^2s -tr -s)(tA-I)-B \\
&=(t^2s -tr -s)tA -(t^2s -tr -s)I-B \numberthis \label{lhs}
\end{align*}

Similarly,

\begin{align*}
B^2AB^2 &= \tr(B^2A) B^2 - A^{-1} \\
&= \tr((sB-I)A)(sB-I) - (tI-A) \\
&= (s\cdot \tr(BA) - \tr(A))(sB-I)-tI+A\\
&=(rs-t)(sB-I)-tI+A \\
&= (rs-t)sB - rsI+A \numberthis \label{rhs}
\end{align*}

\subsection{Irreducible Representations}

According to the group presentation in Corollary \ref{cor:pres}, we must have that $A^2B^{-1}A^2 = B^2AB^2$, and hence, by Lemma \ref{lem:chen} \ref{linind}, if $\rho$ is an irreducible representation then the equality of (\ref{lhs}) and (\ref{rhs}) implies that:

\begin{equation}\label{icond1}
\begin{cases}
1 = (t^2s -tr -s)t\\
-1 = (rs-t)s \\
rs =(t^2s -tr -s)
\end{cases}
\end{equation}
This can be solved in terms of $t$, and hence we have the following:

\begin{lem}
If $\rho: \pi_1(X) \rightarrow \SU(1,1)$ is an irreducible representation with $\tr(\rho(a))=t$, $\tr(\rho(b)) = s$, and $\tr(\rho(ab))=r$ then
\begin{equation}\label{icond2}
\begin{cases}
s = \frac{t}{t^2-1}\\
r = 1-\frac{1}{t^2}
\end{cases}
\end{equation}
\end{lem}

Now let us examine the ``intersection" between reducible and irreducible representations; this will correspond to the points at which reducible representations may be deformed to irreducible ones.  In the notation of section \ref{sec:red}, the traces of reducible representations satisfy:
\begin{align*}
\tr(A) &= e^{5i\theta} + e^{-5i\theta}\\
\tr(B) &= e^{3i\theta} + e^{-3i\theta}\\
\tr(AB) &= e^{8i\theta} + e^{-8i\theta}
\end{align*}
Substituting $z=e^{i\theta}$ and $t,s,r$ for the traces gives:
\begin{align*}
t &= z^5+z^{-5}\\
s &= z^3+z^{-3}\\
r &= z^8+z^{-8}
\end{align*}
Hence, by (\ref{icond2}):
\[
\begin{cases}
z^3+z^{-3} = \frac{z^5+z^{-5}}{(z^5+z^{-5})^2-1}\\
z^8+z^{-8} = 1-\frac{1}{(z^5+z^{-5})^2}
\end{cases}
\]
which is equivalent to:
\[
\begin{cases}
z^{13}+z^{-7}+z^3 + z^7+z^{-13}+z^{-3}- z^5-z^{-5}=0\\
z^{18}+z^{-2}+2z^8+z^2+z^{-18}+2z^{-8} - z^{10}-z^{-10}-1=0
\end{cases}
\]
which one can factor as:
\[
\begin{cases}
(z^{-1}+z)(z^{-2}+z^2)(z^{-10}-z^{-8}+z^{-4}-z^{-2}+1-z^{2}+z^{4}-z^{8}+z^{10})=0\\
(z^{-8}+z^{-6}+z^{-4}-1+z^4+z^6+z^8)(z^{-10}-z^{-8}+z^{-4}-z^{-2}+1-z^{2}+z^{4}-z^{8}+z^{10})=0
\end{cases}
\]
One sees that $z^2$ must be a root of the polynomial $x^{-5}-x^{-4}+x^{-2}-x^{-1}+1-x+x^{2}-x^{4}+x^{5},$ which is the Alexander polynomial of the knot $P(-2,3,7).$  This is consistent with the general result of Theorem \ref{thm:hp}, which implies that reducible representations can be deformed into irreducible ones only if $z^2$ is a root of the Alexander polynomial.

Let us consider consider a deformation of the representation corresponding to the root $x=e^{i\theta_0}$ with $\theta_0 \approx 2.0453 \approx 0.3255(2\pi)$.  This corresponds to $t_0=2\cos(\frac{5\theta_0}{2})\approx 0.78$
\subsection{Realization of Representations}\label{sec:real}

Let us investigate a partial converse of the previous lemma.  In particular, given real numbers $t$, $s$ ,and $r$, when does there exist a representation $\rho: \pi_1(X) \rightarrow \SU(1,1)$ satisfying $\tr(\rho(a))=t$, $\tr(\rho(b)) = s$, and $\tr(\rho(ab))=r$?  Let us restrict ourselves to the case where $\rho(a)$ is elliptic.  Then, we may suppose, by conjugating appropriately if necessary, that:
\begin{align*}
\rho(a) &= \left(
\begin{matrix}
e^{i\alpha} & 0 \\
0 & e^{-i\alpha} 
\end{matrix}\right)\\
\rho(b) &= \left(
\begin{matrix}
\sqrt{1+R^2}e^{i\beta} & R \\
R & \sqrt{1+R^2}e^{-i\beta} 
\end{matrix}\right)\\
\rho(ab) &= \left(
\begin{matrix}
\sqrt{1+R^2}e^{i(\alpha+\beta)} & Re^{i\alpha} \\
Re^{-i\alpha} & \sqrt{1+R^2}e^{-i(\alpha+\beta)} 
\end{matrix}\right)
\end{align*}
for some $\alpha,\beta,R \in \mathbb{R}$.
Hence we are looking for real solutions of:
\begin{align*}
t = \tr(\rho(a)) &= 2\cos\alpha\\
s = \tr(\rho(b)) &= 2\sqrt{1+R^2}\cos\beta\\
r = \tr(\rho(ab)) &= 2\sqrt{1+R^2}\cos(\alpha+\beta)
\end{align*}
For $s\neq0$ we can divide the last two equations:
\begin{align*}
\frac{r}{s} &= \frac{\cos(\alpha+\beta)}{\cos\beta}\\
&= \cos\alpha - \sin\alpha \tan\beta
\end{align*}

As $|t|<2$, $\cos\alpha \neq 1$ so that $\sin\alpha \neq 0$.  Hence, under these conditions:
\begin{align*}
\cos\alpha &= \frac{t}{2}\\
\tan\beta &= \frac{1}{\sin\alpha}\left(\frac{t}{2}-\frac{r}{s}\right)
\end{align*}
which determine $\alpha$ and $\beta$ (up to sign or multiples of $\pi$).  Finally, we can see that such a representation can be constructed as long as there is an $R$ satisfying:
\begin{equation}\label{r_cond}
R^2 = \left(\frac{s}{2\cos\beta}\right)^2-1
\end{equation}
which happens exactly when
\begin{equation}\label{condSU}
\left|\frac{s}{2\cos\beta}\right| \geq 1.
\end{equation}

Notice that when the above inequality fails, one can find a purely imaginary $R$ that satisfies the condition (\ref{r_cond}).  In this case, the representation is actually in $\SU(2)$ instead.  This is consistent with Theorem \ref{thm:hp}, which implies that deformations of reducible representations will be in $\SU(1,1)$ on one ``side" and in $\SU(2)$ on the other. If we further suppose $t,s,r$ satisfy condition \ref{icond2}, then we see that $0<t<1$ implies $s \neq 0$ and that $s$ varies continuously.  Therefore, for a continuous path of $t$-values within this interval, one can produce a continuous path of representations into $\SU(2) \cup \SU(1,1).$

\section{Constructing a Path in the Translation Extension Locus}
\subsection{Trace Computations}
Here we use the calculations of the previous section to demonstrate the existence of a certain path in the translation extension locus.  Notice that, although the previous section considered representations into $\SU(1,1),$ that group is conjugate to $\SL_2(\mathbf{R}).$  Moreover, it suffices to produce a path of representations into $\SL_2(\mathbf{R})$ as such a path will lift to a path of representations into $\tilde{G},$ the universal cover.

By Corollary \ref{cor:pres}, $\rho(\mu) = A^{-1}B^2$.  Hence, we compute:

\begin{align*}
\tr(\rho(\mu)) &= \tr((tI-A)(sB-I))\\
&= ts\cdot \tr(B) - t\cdot \tr(I) - s\cdot \tr(AB) + \tr(A)\\
&= ts^2 - sr - t
\end{align*}

Similarly, we have that $\rho(\sigma) = AB^{-1}A^2B^{-1}A^2$, and hence, using (\ref{lhs}),
\begin{align*}
\tr(\rho(\sigma)) &= \tr((AB^{-1})((t^2s -tr -s)tA -(t^2s -tr -s)I-B))\\
&= (t^2s -tr -s)t\cdot \tr(AB^{-1}A) - (t^2s -tr -s)\tr(AB^{-1}) - \tr(A)\\
&= (t^2s -tr -s)t\cdot \tr(\tr(AB^{-1})A-(B^{-1})^{-1}) - (t^2s -tr -s)\tr(AB^{-1}) - \tr(A)\\
&= (t^2s -tr -s)t\cdot \tr(\tr(A(sI-B))A-B) - (t^2s -tr -s)\tr(A(sI-B)) - \tr(A)\\
&= (t^2s -tr -s)t\cdot \tr((st-r)A-B) - (t^2s -tr -s)(st-r) - t\\
&= (t^2s -tr -s)t((st-r)t-s)- (t^2s -tr -s)(st-r) - t\\
&= (t^2s -tr -s)^2 t- (t^2s -tr -s)(st-r) - t
\end{align*}

Using the conditions in (\ref{icond1}) and (\ref{icond2}), these may be expressed in terms of $t$:
\begin{align}
\tr(\rho(\mu)) &= \frac{t^3}{(t^2-1)^2}-\frac{1}{t}-t \eqqcolon m(t) \label{mu}\\
\tr(\rho(\sigma)) &= \frac{2}{t} - \frac{t}{t^2-1} -\frac{1}{t^3} -t \eqqcolon \ell(t) \label{lamb}
\end{align}

\begin{lem}\label{lem:mono}
For $t \in (0,1),$ $m(t)$ and $\ell(t)$ are strictly increasing as functions of $t.$
\end{lem}
\begin{proof}
One computes, by (\ref{mu}):
\[
m'(t) = \frac{3t^2(t^2-1)-4t^4}{(t^2-1)^3} +\frac{1}{t^2}-1
\]
The first term on the right-hand-side is seen to be positive for positive $t<1,$ and $\frac{1}{t^2}>1$ for such values of $t$ as well.

Similarly, by (\ref{lamb}),
\begin{align*}
\ell'(t) &= -\frac{2}{t^2} - \frac{t^2-1-2t^2}{(t^2-1)^2} +\frac{3}{t^4} - 1 \\
&= \frac{1+t^2}{(t^2-1)^2} + \frac{(t^2+3)(1-t^2)}{t^4}
\end{align*}
which is seen to be positive for $0<t<1.$
\end{proof}

\begin{lem}\label{lem:tan}
If $s$ and $r$ satisfy condition (\ref{icond2}), then on the interval $t\in \left(\sqrt{2-\sqrt{2}},\sqrt{\sqrt{10}-2}\right)$ the quantity $\frac{t}{2}-\frac{r}{s}$ is positive and increasing as a function of $t$.
\end{lem}
\begin{proof}
It is straightforward to compute that 
\begin{align*}
\frac{t}{2}-\frac{r}{s} &= -\left(\frac{t}{2}-\frac{2}{t}+\frac{1}{t^3}\right) \\
&= -\frac{\left(t^2-2\right)^2-2}{2t^3}
\end{align*}
which is easily seen to be positive whenever $-\sqrt{2}<t^2-2 < \sqrt{2},$ and so in particular, it is positive when $t\in \left(\sqrt{2-\sqrt{2}},\sqrt{2+\sqrt{2}}\right).$  Now, differentiating with respect to $t$ gives
\begin{align*}
\frac{\mathrm{d}}{\mathrm{d}t}\left(\frac{t}{2}-\frac{r}{s}\right) &=  -\left(\frac{1}{2}+\frac{2}{t^2}-\frac{3}{t^4}\right) \\
&= -\frac{\left(t^2+2\right)^2-10}{2t^4}
\end{align*}
which is positive whenever  $-\sqrt{10}<t^2+2 < \sqrt{10},$ and this is satisfied when $t\in \left(0,\sqrt{\sqrt{10}-2}\right).$  The interval in the statement of the lemma is the intersection of these two intervals.
\end{proof}

\subsection{A Path of Irreducible Representations}

Now observe that, on the interval $t\in(t_0,t_1),$ $\alpha$ is increasing as a function of $t$ since $\alpha$ remains between $\frac{3\pi}{2}$ and $2\pi.$  This implies that $\sin\alpha$ is negative and increasing on this interval.  Hence, by lemma \ref{lem:tan}, $\tan\beta$ is negative and decreasing on this interval.  Because $\beta$ is in the second quadrant, it follows that $\cos\beta$ is negative and increasing.  Notice that $s = \frac{t}{t^2-1}$ is negative and decreasing on this interval (since $t_1<1$), and so the quantity $\frac{s}{2\cos\beta}$ is positive and increasing on the interval.  Therefore, condition (\ref{condSU}) is satisfied on the interval, and we have shown:

\begin{thm}\label{thm:rep}
There exists a continuous path of elliptic representations $\rho(t): \pi_1(X) \rightarrow \SU(1,1)$ for $t\in [t_0,t_1)$ which is irreducible except at $t_0$ and is such that $\rho(t_0)=\rho_0$ and $\tr(\rho(t)(a))=t.$
\end{thm}

Now observe that, since $\mu$ and $\sigma$ commute, if $\rho$ is elliptic then $\rho(\mu)$ and $\rho(\sigma)$ have the same fixed point.  Hence, up to conjugation, they have the form:
\begin{align*}
\rho(\mu) &= \left(\begin{matrix}
e^{i\phi} & 0 \\
0 & e^{-i\phi}
\end{matrix}\right) \\
\rho(\sigma) &= \left(\begin{matrix}
e^{i\psi} & 0 \\
0 & e^{-i\psi}
\end{matrix}\right)
\end{align*}
So  that $\tr(\rho(\mu)) = 2\cos\phi$ and $\tr(\rho(\sigma)) = 2\cos\psi$ , and also $\trans(\rho(\mu)) = \frac{\phi}{\pi}$.

Moreover,
\begin{align*}
\rho(\lambda) &= \rho(\mu^{-19}\sigma)\\
&= \left(\begin{matrix}
e^{i(-19  \phi + \psi)} & 0 \\
0 & e^{i(19\phi - \psi)}
\end{matrix}\right)
\end{align*}
Hence $\trans(\rho(\lambda)) = \frac{-19  \phi + \psi}{\pi}$.  So the point on the translation extension locus is: $(\frac{\phi}{\pi},\frac{-19  \phi + \psi}{\pi})$.  Note that, when considering the above matrices in $\SU(1,1)\cong \SL_2(\mathbf{R}),$ the values of $\phi$ and $\psi$ are only relevant modulo $2\pi.$  However, since the extension translation locus is constructed from lifted representations to $\widetilde{G},$ it will be important to keep track of the actual values.

\begin{proof}[Proof of Theorem \ref{thm:main}]
Let us investigate the part of the extension translation locus coming from the path of representations given by Theorem \ref{thm:rep}.  At $t=t_0,$ we begin with a reducible representation with coordinates $(\frac{\theta_0}{2\pi},0)$.  Hence, at this point, $\phi \approx 1.02$ and $\psi \approx 19.38 \approx 6\pi + .53$.  Using the notation of section \ref{sec:real}, we have that, at this point, $\alpha = \frac{5\theta_0}{2} \approx 5.11$ and $\beta = \frac{3\theta_0}{2} \approx 3.07.$
As $t \rightarrow t_1,$ lemma \ref{lem:mono} tells us that $m(t)$ increases monotonically to 2, and hence $\phi$ decreases monotonically to 0. Similarly, $\ell(t)$ also increases monotonically to 2 so that $\psi$ decreases monotonically to $6\pi.$ Therefore, the limiting parabolic point in the translation extension locus is $(6,0),$ and in fact there is a path in the translation extension locus connecting $(\frac{\theta_0}{2\pi},0)$ and $(6,0)$ which is contained within $\left\lbrace (\mu^*,\lambda^*)\in\mathbf{R}^2:0\leq \mu^* \leq \frac{\theta_0}{2\pi}\right\rbrace$ (see Figure \ref{fig:et}). By the symmetries of the translation extension locus (Theorem \ref{thm:el}), there is also a continuous path connecting the points $(1-\frac{\theta_0}{2\pi},0)$ and $(1,-6)$ contained within $\left\lbrace (\mu^*,\lambda^*)\in\mathbf{R}^2:1-\frac{\theta_0}{2\pi} \leq \mu^* \leq 1\right\rbrace.$ It is easy to see that all lines through the origin of slope greater than -6 will intersect one of these paths.  By construction, these paths consist of elliptic points, and so, by Theorem \ref{thm:int}, the Dehn filling $M(r)$ for $r\in(-\infty, 6)$ will be orderable whenever it is irreducible. As the exceptional slopes of the knot $K,$ as given in Problem 1.77 of \cite{KirbyList} all lie outside this interval, all Dehn fillings within the interval are hyberbolic and hence irreducible.  Therefore, they are all orderable.
\end{proof}

\begin{figure}
\centering
\includegraphics[width=.5\textwidth]{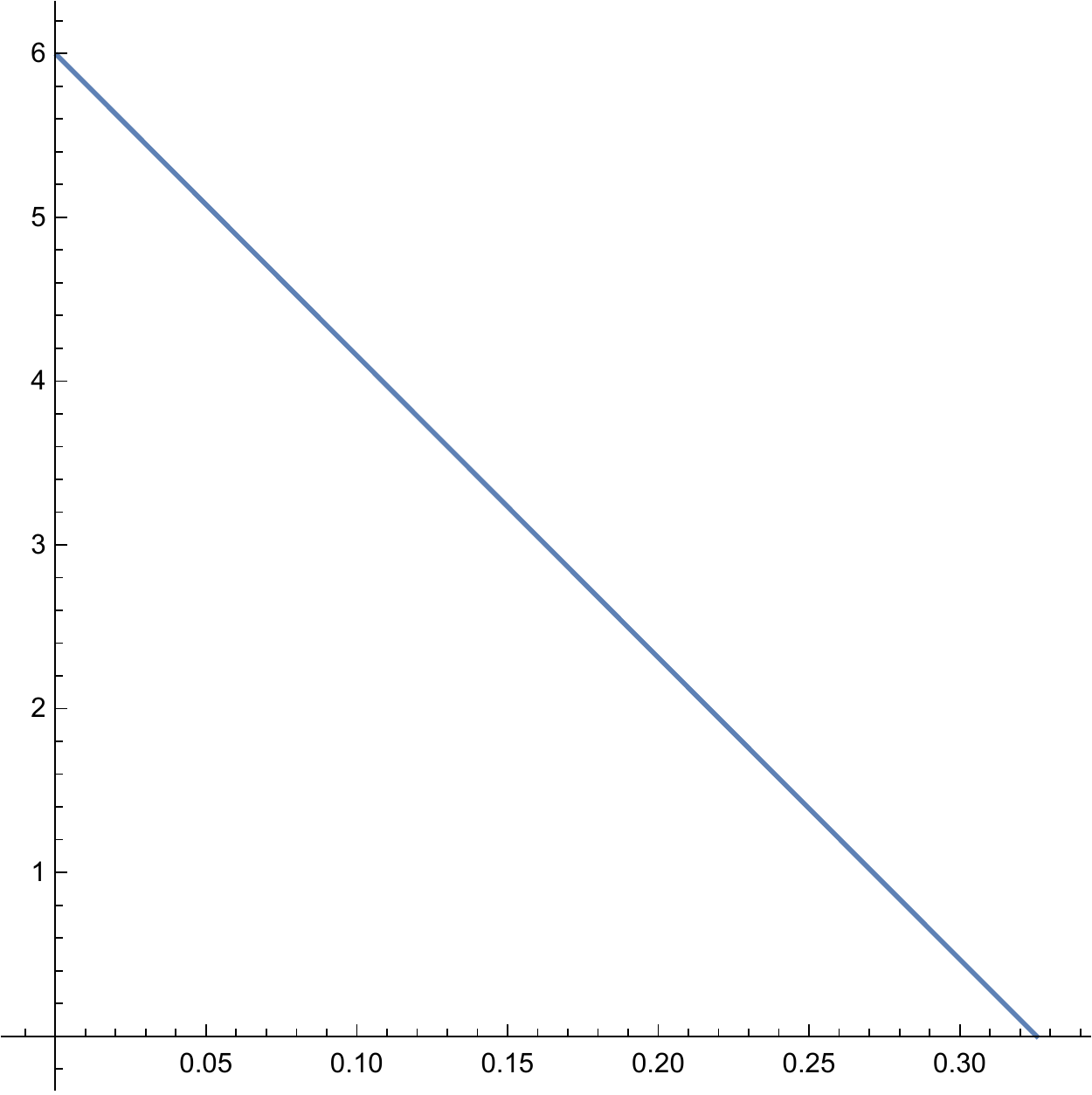}
\caption{The translation extension locus for the 1-parameter family of representations constructed above.}
\label{fig:et}
\end{figure}

\begin{proof}[Remark.]\let\qed\relax
Figure \ref{fig:et} shows a plot of the part of the translation extension locus corresponding to the path in Theorem \ref{thm:rep}, computed using Mathematica \cite{WM}. One sees that the curve in the extension translation locus actually remains within the first quadrant and resembles a straight line (as did the analogous curve in \cite{CD}.  However, explicit computation shows that it is not actually a line. Indeed, figure \ref{fig:slp} shows how the slope of this curve varies with the parameter $t.$  It is interesting to note that it ranges between -18.4 to -18.5, suggesting that the meridional ($-19\phi$) term  in $\frac{-19  \phi + \psi}{\pi}$ is dominant, with the ``deviation" coming from $\psi$ contributing less than 1 to the slope.  Indeed, the joint monotonicity of $m(t)$ and $\ell(t)$, as given by Lemma \ref{lem:mono} essentially shows that the total ``deviation" of the endpoint from that of a line of slope -19 must be less than one, as the total displacement of $\psi$ must be less than one.

\begin{figure}
\centering
\includegraphics[width=.75\textwidth]{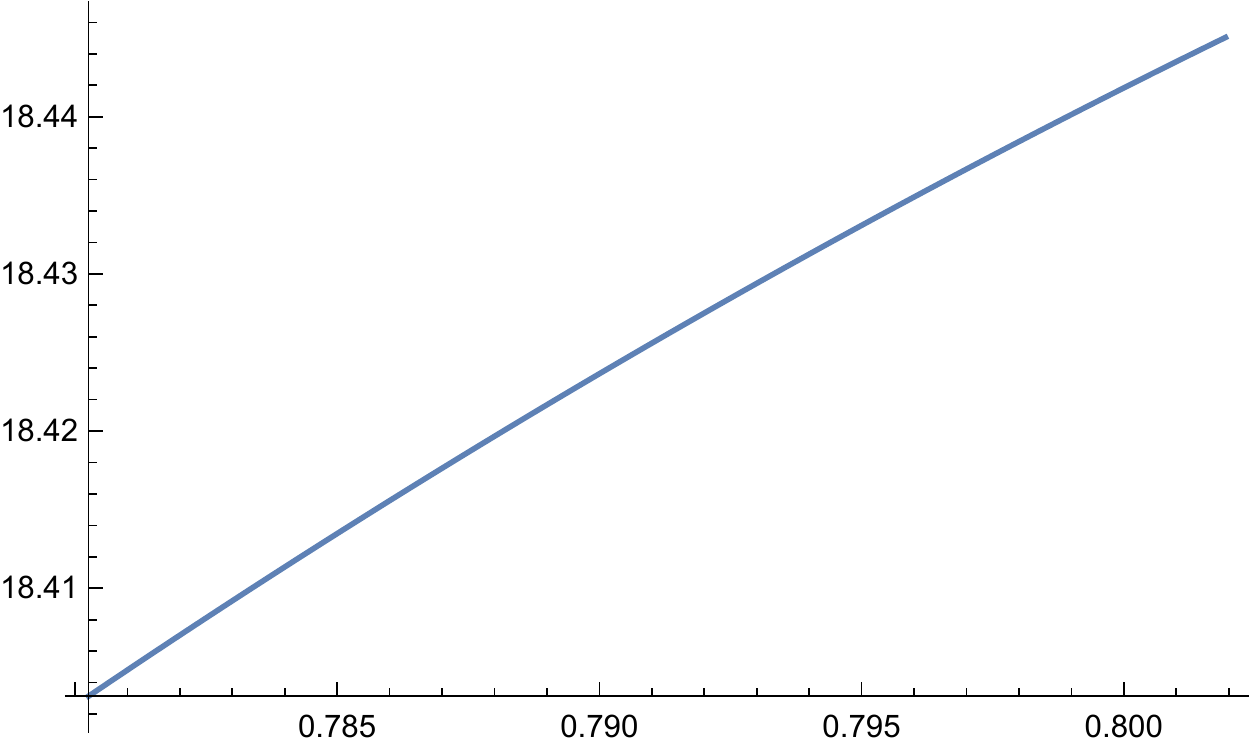}
\caption{The slope of the curve of the translation extension locus constructed in the proof of Theorem \ref{thm:main}, as a function of $t.$}
\label{fig:slp}
\end{figure}

Also using Mathematica, one can plot the translation extension locus of this knot coming from all roots of the Alexander polynomial (hence, giving a picture of two fundamental domains of the entire locus).  This is shown in Figure \ref{fig:etF}, and it agrees with the corresponding diagram obtained by Culler and Dunfield's (see Figure 3 of \cite{CD}
; the reason for the vertical reflection is most likely a difference in convention for the orientation of the meridian-longitude pair).
\end{proof}

\begin{figure}
\centering
\includegraphics[width=.75\textwidth]{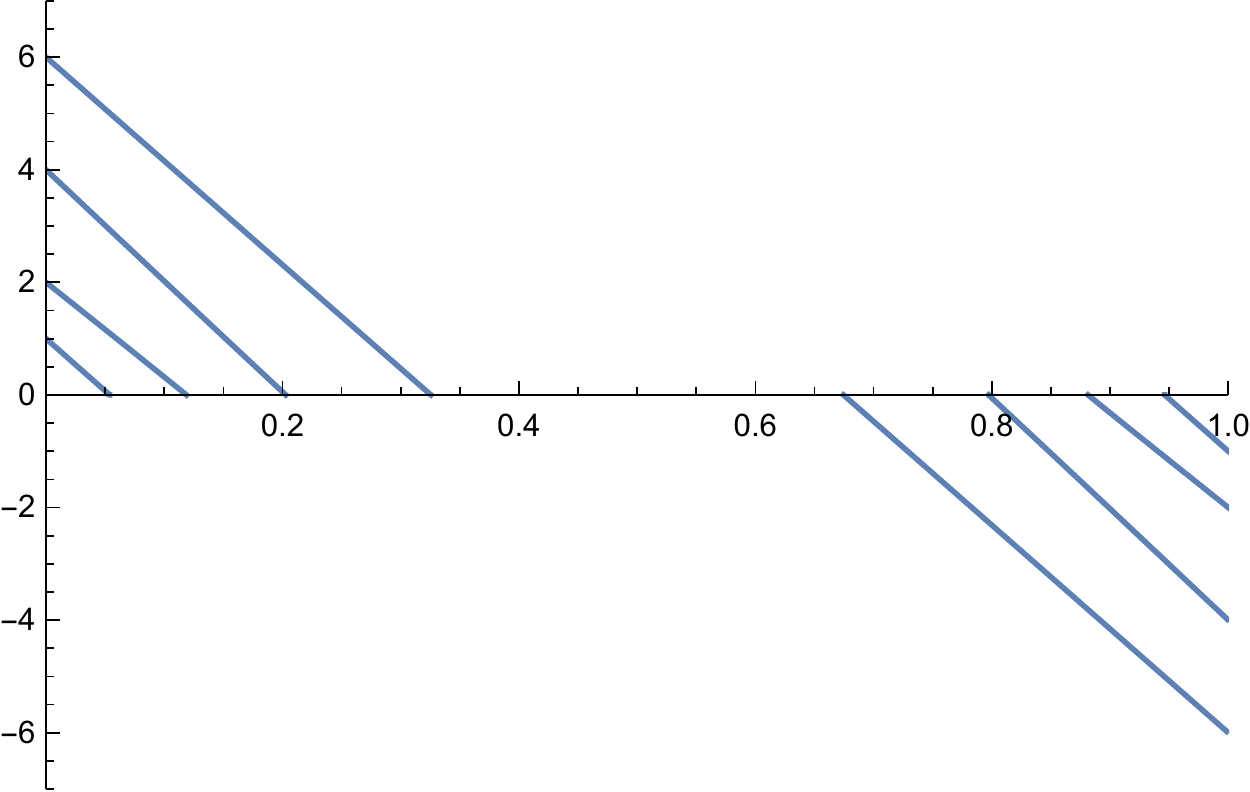}
\caption{A (double) fundamental domain of the translation extension locus for the $(-2,3,7)$-pretzel knot}
\label{fig:etF}
\end{figure}

\section{Generalizations}\label{sec:gen}

\subsection{A Family of Twisted Torus Knots}

We now consider the family of twisted torus knots $T^1_{3,3k+2}$ for $k\geq 1,$ which contains the $(-2,3,7)$-pretzel knot as the special case $k=1.$  By Lemma \ref{lem:genpres}, the fundamental group of the knot exterior has presentation:
\[
\pi_{3,3k+2}^1 \coloneqq \left\langle a, b \left| a^2 b^{-k}a^2 = b^{k+1} a  b^{k+1} \right. \right\rangle
\]

As in section \ref{sec:reps}, we see that if a representation $\rho: \pi_{3,3k+2}^1 \to \SL_2(\mathbf{R})$ is such that $\rho(a)=A$ and $\rho(b)=B,$ then using Lemma \ref{lem:chen}, we see that:

\begin{align*}
A^2B^{-k}A^2 &= \tr(A^2B^{-k})A^2 - B^k \\
&= \tr(A^2B^{-k})(tA-I) - (\omega_{k}(s)B-\omega_{k-1}(s)I) \numberthis \label{LHS}
\end{align*}
where $t=\tr(A)$, $s=\tr(B)$, and $r=\tr(AB).$  Similarly:
\begin{align*}
B^{k+1}AB^{k+1} &= \tr(B^{k+1}A)B^{k+1} - A^{-1} \\
&= \tr(\omega_{k+1}(s)B-\omega_{k}(s)I)B^{k+1} - (tI-A) \\
&= (r\omega_{k+1}(s)-t\omega_{k}(s))(\omega_{k+1}(s)B-\omega_{k}(s)I)-tI+A \numberthis \label{RHS}
\end{align*}
Equating (\ref{LHS}) and (\ref{RHS}), we see by Lemma \ref{lem:chen}\ref{linind}, if  $\rho$ is irreducible then:
\begin{equation}\label{gprecond}
\begin{cases}
\tr(A^2B^{-k}) = \frac{1}{t} \\
-\omega_{k}(s) = (r\omega_{k+1}(s)-t\omega_k(s))\\
-\tr(A^2B^{-k})+\omega_{k-1}(s) = -(r\omega_{k+1}(s)-t\omega_k(s))\omega_k(s)-t
\end{cases}
\end{equation}
From this we readily find that the following must hold:
\begin{equation}\label{gencond}
\begin{cases}
t-\frac{1}{t} = -\omega_{k-1}(s)+\frac{\omega_k(s)^2}{\omega_{k+1}(s)}\\
r=\frac{t\omega_k(s)}{\omega_{k+1}(s)}-\frac{\omega_k(s)}{\omega_{k+1}(s)^2}
\end{cases}
\end{equation}
Since the irreducible characters are one-dimensional, it follows that these must be all the relations.  Notice also that, as long as $\omega_{k+1}(s)\neq 0$, each value of $s$ corresponds to exactly two values of $t$ (each being the negative reciprocal of the other) and hence two values of $r$.

We may write the traces of the images of the meridian and (surface-framed) longitude in terms of these parameters.  By Lemma \ref{lem:genpres}, $\mu = a^{-1}b^{k+1}$ and so:
\begin{align*}
\tr(\rho(\mu)) &= \tr(A^{-1}B^{k+1})\\
&= \tr((tI-A)B^{k+1})\\
&=t(s\omega_{k+1}(s)-2\omega_k(s))-(r\omega_{k+1}(s)-t\omega_k(s))\\
&=t(s\omega_{k+1}(s)-2\omega_k(s))+\frac{\omega_k(s)}{\omega_{k+1}(s)}
\end{align*}
where the condition (\ref{gencond}) was used in the last step.

Similarly, by Lemma \ref{lem:genpres}, we know that $\sigma = ab^{-k}a^2b^{-k}a^2$ so that:
\begin{align*}
\tr(\rho(\sigma)) &= \tr(AB^{-k}A^2B^{-k}A^2) \\
&=  \tr(AB^{-k}(\tr(A^2B^{-k})A^2-B^k)) \\
&=\tr(A^2B^{-k})\tr(AB^{-k}A^2)-\tr(A) \\
&=\tr(A^2B^{-k})\tr((\tr(AB^{-k})A-B^k)A)-t \\
&=\tr(A^2B^{-k})(\tr(AB^{-k})\tr(A^2)-\tr(B^kA))-t \\
&=\tr(A^2B^{-k})(\omega_{-k}(s)r-\omega_{-k-1}(s)t)(t^2-2)-(\omega_k(s)r-\omega_{k-1}(s)t))-t \\
&=\frac{1}{t}(\omega_{k+1}(s)t-\omega_{k}(s)r)(t^2-2)-(\omega_k(s)r-\omega_{k-1}(s)t))-t
\end{align*}
where the condition (\ref{gprecond}) was used in the last step.

\subsection{Empirical Observations}
Using the calculations of the previous section, one can compute the translation extension locus corresponding to the twisted torus knot $T_{3,3k+2}^1$ for any value of $k$.  Indeed, by \cite{Mor}, we have an explicit formula for the Alexander polynomial of $T_{3,3k+2}^1$, from which the roots may be approximated.  Moreover, for each deformation coming from  a root of the Alexander polynomial, we can verify that equation (\ref{condSU}) holds (or the analogue with the roles of $s$ and $t$ reversed), hence confirming that these come from $\SL_2(\mathbf{R})$ representations.  Figure \ref{fig:ELs} shows the graphs of these extension translation loci for $k=2,3,4,$ produced using Mathematica.

\begin{figure}
\centering
\begin{subfigure}{.3\textwidth}
\includegraphics[width=\textwidth]{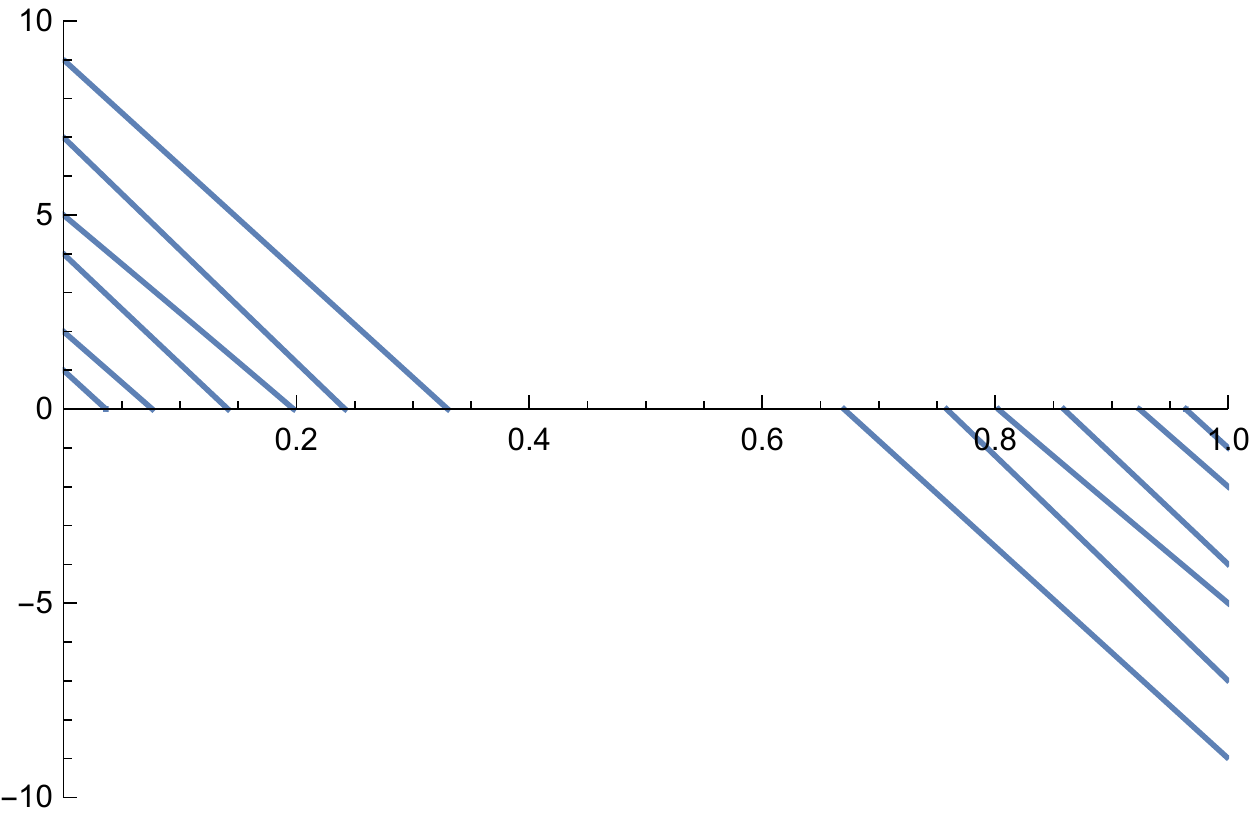}
\caption{$k=2$}
\end{subfigure}
\begin{subfigure}{.3\textwidth}
\includegraphics[width=\textwidth]{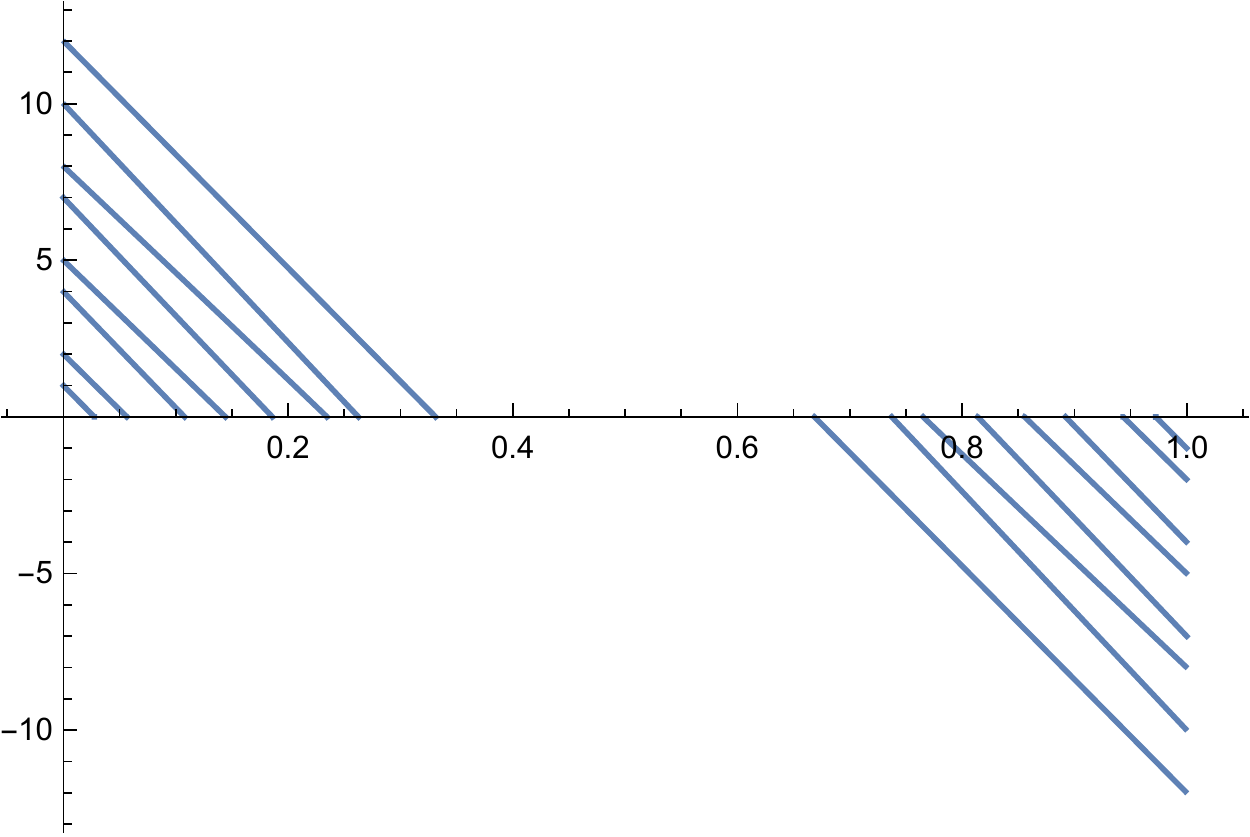}
\caption{$k=3$}
\end{subfigure}
\begin{subfigure}{.3\textwidth}
\includegraphics[width=\textwidth]{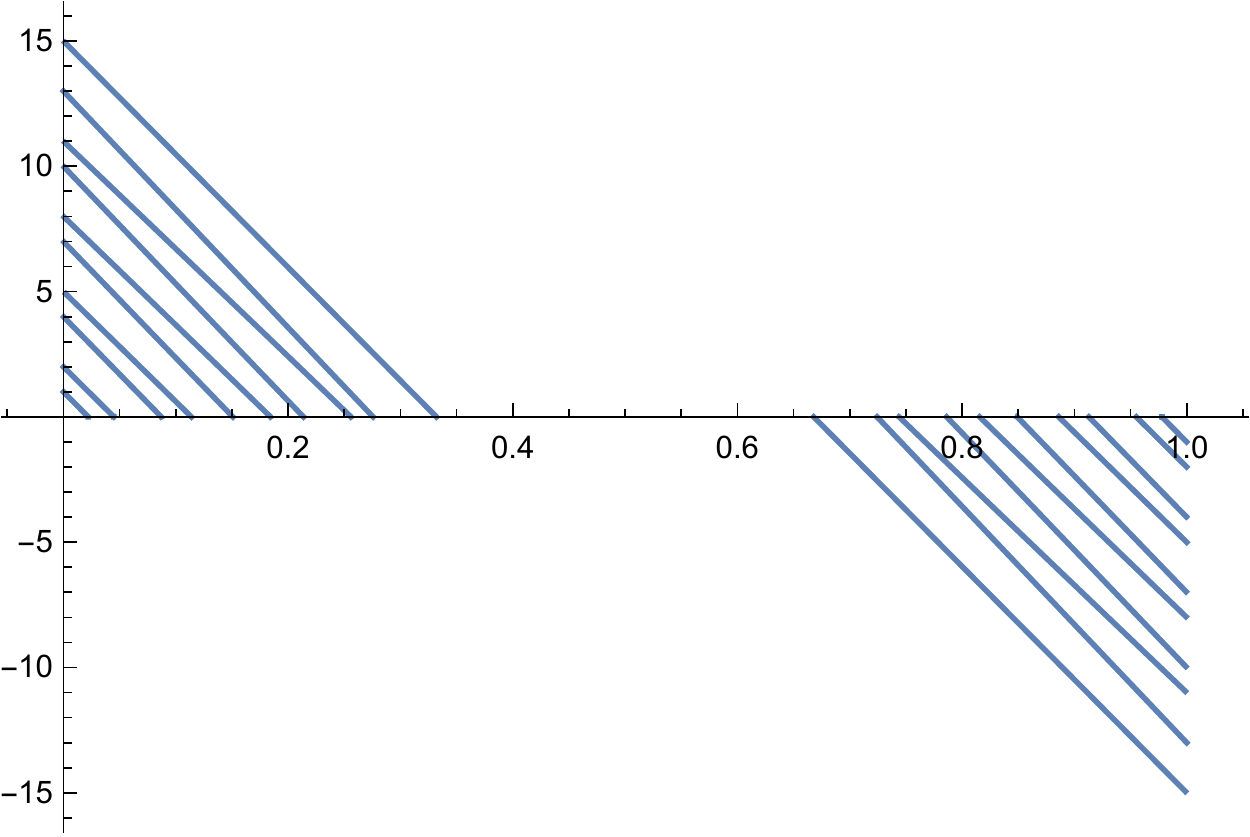}
\caption{$k=4$}
\end{subfigure}
\caption{Double fundamental domains of extension translation loci for the knots $T_{3,3k+2}^1$ for $k=2,3,4.$}
\label{fig:ELs}
\end{figure}

From these we can make the following observations:
\begin{enumerate}[label=(\roman*)]
\item The translation extension locus corresponding to $T_{3,3k+2}^1$ seems to consist of several nearly parallel ``almost lines", each with a ``slope" approximately $3(3k+2)+4.$ \label{conj:line}
\item The ``lines" corresponding to roots of the Alexander polynomial with argument in $(0,\pi)$ seem to lie above the horizontal axis (hence, by the symmetries of the translation extension locus, those with argument in $(\pi,2\pi)$ lie below the axis). \label{conj:sep}
\item The longest ``line" seems to correspond to a root of the Alexander polynomial with argument approximately $\frac{2\pi}{3}.$ \label{conj:third}
\item The maximum height of the translation extension locus seems to be achieved by the point $(0,3k+3)$ \label{conj:height}
\end{enumerate}

\ref{conj:line} was also noticed by Culler and Dunfield in \cite{CD}, who asked if translation extension loci for all twisted torus knots (and also Berge knots) have this form.

Notice that \ref{conj:height} would imply that all surgery slopes in the interval $(-\infty,3k+3)$ yield orderable manifolds (as long as they are irreducible).  On the other hand, notice that the Seifert genus of $T_{3,3k+2}^1$ is $3k+2.$  By the result of \cite{OSz}, the non-L-space surgeries are those in the interval $(-\infty,6k+3).$  Moreover, Tran has shown that surgeries in that interval do in fact yield non-orderable manifolds \cite{Tran}.  Hence, it seems that this method of constructing left-orderings leaves out the slopes in the interval $[3k+3,6k+3)$ as still unconfirmed vis-\`{a}-vis Conjecture \ref{conj:LS}.

Moreover, \ref{conj:height} is roughly a consequence of \ref{conj:line} and \ref{conj:third} since if those two hold, then one would expect the longest ``line" which starts near $(\frac{1}{3},0)$ to reach the vertical axis near $(0,3k+2+\frac{4}{3})$.  The nearest point on the integer lattice is $(0,3k+3)$ (notice that a representation corresponding to such a point is parabolic, and parabolic elements of $\tilde{G}$ must have integer translation).

More precisely, suppose the analogue of Lemma \ref{lem:mono} holds for an appropriate interval (in particular, an interval corresponding to one of the ``lines").  Using the fact that points on the translation extension locus have the form $(\frac{\phi}{\pi},\frac{-(3(3k+2)+4)  \phi + \psi}{\pi})$, one sees that if the ``line" has one endpoint at $(x,0)$, then $\psi=(3(3k+2)+4)\pi x$ at that endpoint.  At the other endpoint, $\phi=0,$ so that the endpoint will be $(0,\frac{\psi}{\pi}),$ where $\psi \geq ((3(3k+2)+4)x-1)\pi,$ by monotonicity.  Now if \ref{conj:third} holds, the longest ``line" has one endpoint at $\approx (1/3,0)$ so that the maximal height is at least $3k+2+1/3$ (and similarly, monotonicity implies this height will be at most $3k+4+1/3.$  As mentioned above, the height must be an integer, and so obtain the maximal height either $3k+3$ or $3k+4.$

For $k\geq2,$ it is not the case that the meridional and longitudinal traces are monotonic on the same intervals.  Nevertheless, for the examples computed above, an analogue of Lemma \ref{lem:mono} does appear to hold for all the intervals corresponding to ``lines" in the translation extension locus.

\printbibliography

\end{document}